\documentclass[12pt, reqno]{amsart}
\usepackage{amsmath, amsthm, amscd, amsfonts, amssymb, graphicx, color}
\usepackage[bookmarksnumbered, colorlinks, plainpages]{hyperref}
\hypersetup{colorlinks=true,linkcolor=red, anchorcolor=green, citecolor=cyan, urlcolor=red, filecolor=magenta, pdftoolbar=true}

\newtheorem{theorem}{Theorem}[section]
\newtheorem{lemma}[theorem]{Lemma}

\newtheorem{corollary}[theorem]{Corollary}
\theoremstyle{definition}
\newtheorem{definition}[theorem]{Definition}

\newtheorem{property}[theorem]{Property}
\theoremstyle{remark}
\newtheorem{remark}[theorem]{Remark}
\numberwithin{equation}{section}

\begin{document}

\setcounter{page}{1}

\title[Local and blowing-up solutions ...]{Local and blowing-up solutions for an integro-differential diffusion equation and system}

\author[M. Borikhanov, B. T. Torebek]{Meiirkhan Borikhanov, Berikbol T. Torebek}

\address{{Meiirkhan Borikhanov \newline Al--Farabi Kazakh National University \newline Al--Farabi ave. 71, 050040, Almaty, Kazakhstan \newline Institute of
Mathematics and Mathematical Modeling \newline 125 Pushkin str.,
050010 Almaty, Kazakhstan.}}
\email{{borikhanov@math.kz}}
\address{{Berikbol T. Torebek \newline Al--Farabi Kazakh National University \newline Al--Farabi ave. 71, 050040, Almaty, Kazakhstan \newline Institute of
Mathematics and Mathematical Modeling \newline 125 Pushkin str.,
050010 Almaty, Kazakhstan \newline Department of Mathematics: Analysis, Logic and Discrete Mathematics \newline
Ghent University, Belgium}}
\email{{torebek@math.kz, berikbol.torebek@ugent.be}}

%\dedicatory{This paper is dedicated to Professor ABCD}
\thanks{The second author was supported by the FWO Odysseus Project and Ministry of Education and Science of the Republic of Kazakhstan Grant AP05131756. No new data was collected or generated during the course of research.}

\let\thefootnote\relax\footnote{$^{*}$Corresponding author}



\begin{abstract}In the present paper initial problems for the semilinear integro-differential diffusion equation and system are considered. The analogue of Duhamel principle for the linear integro-differential diffusion equation is proved. The results on existence of local mild solutions and Fujita-type critical exponents to the semilinear integro-differential diffusion equation and system are presented.
\end{abstract}
\maketitle
%\tableofcontents
\section{Introduction and statement of problem}
The main goal of the present paper is to obtain results on local existence and global non-existence for the integro-differential diffusion equation
\begin{equation}\label{1.1}{{u}_{t}}(x,t)= {D}_{0+,t}^{1-\alpha }\Delta_xu(x,t)+f(x,t,u), (x,t)\in \mathbb{R^N} \times \left( 0,T \right)=\Omega_T, \end{equation}
with the initial condition
\begin{equation}\label{1.2}u\left( x,0 \right)={{u}_{0}}\left( x \right)\ge 0,\end{equation}
where ${D}_{0+,t}^{\alpha }$ is the Riemann-Liouville fractional derivative of order $\alpha \in(0,1)$ \cite{Kilbas}.

Also, we consider questions on local existence and global non-existence  for the integro-differential diffusion system
\begin{equation}\label{1.3}\left\{\begin{array}{l}
{u}_{t}( x,t )-{D_{0+,t}^{1-\alpha }\Delta_x{u}( x,t )}=f(x,t,v) \,\,\,\,\text{ in}\,\,\,\ (x,t)\in \mathbb{R^{N}}\times(0,T)={\Omega}_{T},\\{}\\
{v}_{t}( x,t )-{D_{0+,t}^{1-\beta }\Delta_x{v}( x,t )}=g(x,t,u) \,\,\,\,\text{ in}\,\,\,\ (x,t)\in \mathbb{R^{N}}\times(0,T)={\Omega}_{T},\end{array}\right.
\end{equation}
subject to the initial conditions
\begin{equation}\label{1.4}u\left( x,0 \right)={{u}_{0}}\left( x \right)\ge ~0,\text{   }v\left( x,0 \right)={{v}_{0}}\left( x \right)~\ge 0,\text{   }x\in \mathbb{R^{N}},\end{equation}
where $0<\alpha ,\beta <1.$

In the case $\alpha=0,$ problem \eqref{1.1}-\eqref{1.2} coincides with classical Rayleigh-Stokes problem. In fluid mechanics, the problem to determining the flow created by a sudden movement of a plane from rest is called a Rayleigh-Stokes problem. This is considered as one of the simplest non-stationary problems, which has an exact solution for the Navier-Stokes equations.

In the fractional case, problem \eqref{1.1}-\eqref{1.2} occurs in the study of flows of the Oldroyd-B fluid \cite{Shen, Xue, Ji}. The Oldroyd-B fluid is one of the most important classes for dilute solutions of polymers. We also note that integro-differential diffusion equations of type \eqref{1.1} were studied in \cite{BKT, Chan, Fuj90a, Fuj90b, KiraneT, Ruzh}.

Solutions of initial value problems for non-linear parabolic partial differential equations may not exist for all time. In other words, these solutions may blow up in some sense or an other. Recently, in connection with problems for some class of non-linear parabolic equations, \cite{Kaplan}, \cite{Friedman} and \cite{Ito} gave certain sufficient conditions under which the solutions blow up in a finite time. Although their results are not identical, we can say according to them that the solutions are apt to blow up when the initial values are sufficiently large. The exponents for which solutions can blow up is called critical exponents of Fujita.

In the first instance, Fujita studied the Cauchy problem for the semi-linear diffusion equation in \cite{Fujita1}:
\begin{equation}\label{0.1}\left\{\begin{array}{l}
{{u}_{t}}-\Delta u={{u}^{p+1}},\,\,\,x\in \mathbb{R^{N}}\times \left( 0,T  \right), T\leq+\infty,\\{}\\
u\left( x,0 \right)=u_0\left( x \right)\ge 0,\text{  for  }x\in \mathbb{R^{N}}, \end{array}\right.
\end{equation}
where $p$ is a positive number, $u_0\left( x \right)\in {{L}^{1}}\left( \mathbb{R^{N}} \right)$ is nonnegative, positive on some subset of ${\mathbb{{R}}^{N}}$ of positive measure and $\Delta $ denotes the Laplacian in $N$ variables.

A (classical or weak solution) of equation on $\mathbb{R^{N}}\times \left[ 0,T \right)$ for some $T<+\infty $ is called a local solution. The supremum of all such $T$ for which a solution exists is named the maximal time of existence ${{T}_{\max }}$. When ${{T}_{\max }}=+\infty $ ${{T}_{\max }}<+\infty $ we say the solution is global and when we say that it is not global (or the solution "blows up in finite time"), respectively.

Let ${{p}_{c}}=2/N.$ Fujita proved the following assertions:

(i)	if $0<p<{{p}_{c}}$, $u_0\left( x \right)>0$ for some ${{x}_{0}}$, in this case the solution of problem \eqref{0.1} grows infinitely at some finite instant of time;

(ii)	if $p>{{p}_{c}}$, for each $k>0$, there exists a $\delta >0$ such that problem \eqref{0.1} has a global solution whenever $0\le a\left( x \right)\le \delta {{e}^{-k{{\left| x \right|}^{2}}}}$. The number ${{p}_{c}}$ is referred to as the critical exponent. In the critical case, this problem was solved in \cite{Hayakawa} for $N=1,2$ and in \cite{Kobayashi} for arbitrary $N$. It was shown that if $p={{p}_{c}}$, there is no nonnegative global solution for any nontrivial nonnegative initial data.

Later, Fujita in \cite{Fujita2} extended his own results to the more general case in which $f\left( u \right)$ (the term describing the reaction) is convex and satisfies appropriate conditions (the main of which is the Osgood condition). The results obtained for problem \eqref{0.1} were generalized in \cite{Bandle} for an initial-boundary value problem in a cone with the term ${{\left| x \right|}^{\sigma }}{{u}^{p+1}}$ instead of ${{u}^{p+1}}$. In this case, the critical exponent is equal to $\left( 2+\sigma  \right)/N$.

After that, Qi \cite{Qi} studied the equation
\begin{equation*}\left\{\begin{array}{l}
{{u}_{t}}=\Delta {{u}^{m}}+{{\left| x \right|}^{\sigma }}{{t}^{s}}{{u}^{p}}, t>0,\text{  }x\in \mathbb{R^{N}},\\{}\\
u\left( x,0 \right)=u_0\left( x \right)\ge 0,\text{  for  }x\in \mathbb{R^{N}}, \end{array}\right.
\end{equation*}
and had showed that the critical exponent for this problem is equal to $\left( m-1 \right)\left( s-1 \right)+\left( 2+2s+\sigma  \right)/N>0$.

The following parabolic equation with the fractional power ${{\left( -\Delta  \right)}^{\beta /2}},0<\beta <2$ of the Laplace operator was considered by Sugitani in \cite{Sugitani}:
$${{u}_{t}}+{{\left( -\Delta  \right)}^{\beta /2}}u={{u}^{1+p}}, \left( x,t \right)\in\mathbb{R^{N}} \times\mathbb{{R}^{+}}.$$

Using Fujita's method in \cite{Fujita1}, the authors \cite{Guedda} discussed nonnegative solutions of the equation
\begin{equation}\label{0.2}{{u}_{t}}+{{\left( -\Delta  \right)}^{\beta /2}}u=h\left( t \right){{u}^{1+p}},  \left( x,t \right)\in\mathbb{R^{N}} \times\mathbb{{R}^{+}}.\end{equation}
where $h\left( t \right)$ behaves as  $t^{\sigma}, \sigma >-1, 0<p, \alpha N \le \beta \left( 1+\sigma  \right).$ The proof given in \cite{Guedda} is based on the reduction of Eq.\eqref{0.2} to an ordinary differential equation for the mean value of $u$ with the use of the fundamental solution [say, ${{P}_{\beta }}\left( x,\text{ }t \right)$] of ${{L}_{\beta }}:=\text{ }\partial /\partial t+{{\left( -\Delta  \right)}^{\beta /2}}$. Apparently, the approach of \cite{Guedda} cannot be used for systems of two differential equations with distinct diffusion terms unless, for example, ${{P}_{\beta }}\left( x,\text{ }t \right)$ can be compared with ${{P}_{\gamma }}\left( x,\text{ }t \right)$ for $\beta <\gamma .$ This has been done by Kirane and col. in \cite{Kirane}.

The following spatio-temporal fractional equation
\begin{equation}\label{*}\left\{\begin{array}{l}
\mathcal{D}_{0+}^{\alpha }u+{{\left( -\Delta  \right)}^{\beta /2}}\left( u \right)=h\left( x,t \right){{u}^{1+p}},\,\,\,\,\left( x,t \right)\in \mathbb{R^{N}}\times \mathbb{{R}^{+}},\\{}\\
u\left( x,0 \right)=a\left( x \right)\ge 0,\,\,\,\,x\in \mathbb{R^{N}},\end{array}\right.
\end{equation}
where $\mathcal{D}_{0+,t}^{\alpha }$, for $\alpha \in \left( 0,1 \right)$ is the Caputo fractional derivative and $\beta \in \left[ 1,2 \right]$ with nonnegative initial data was considered in \cite{Kirane}. The critical exponent is equal to $1<p<{{p}_{c}}=1+\frac{\alpha \left( \beta +\alpha  \right)+\beta \rho }{\alpha N+\beta \left( 1-\alpha  \right)}$.

On the other hand, Escobedo and Herrero \cite{Escobedo} considered blowing-up solutions for the semi-linear reaction-diffusion system
\begin{equation*}\left\{\begin{array}{l}
u_t-\Delta u=v^{p}, \,\,\text{  for  }\,\, x\in \mathbb{R^{N}},t>0,\\{}\\
v_t-\Delta v=u^{q}, \,\,\text{  for  }\,\,\, x\in \mathbb{R^{N}},t>0,\end{array}\right.
\end{equation*}
where $N\geq1, p>1, q>1$ and
$$u\left( x,0 \right)={{u}_{0}}\left( x \right)\ge ~0,\text{   }v\left( x,0 \right)={{v}_{0}}\left( x \right)~\ge 0, x\in\mathbb{R^{N}}. $$

They proved that the critical exponent for the case $pq>1$ of this problem is equal to $$\frac{\gamma+1}{pq-1}\geq\frac{N}{2},$$ where $\gamma=\max\{p,q\}$.

Later on, Kirane and al. in \cite{Kirane} considered the following system of spatio-temporal fractional equations
\begin{equation*}\left\{\begin{array}{l}
{\mathcal{D}_{0|t}^{\alpha }u }+(-\Delta)^{\beta/2}u=|v|^{p}, \,\,\text{  for  }\,\,\,( x,t )\in \mathbb{R^{N}}\times\mathbb{R^{+}},\\{}\\
{\mathcal{D}_{0|t}^{\delta }v}+(-\Delta)^{\gamma/2}v=|u|^{q}, \,\,\text{  for  }\,\,( x,t )\in \mathbb{R^{N}}\times\mathbb{R^{+}},\end{array}\right.
\end{equation*}
with Cauchy data
$$u(x,0)=a(x)\geq0,\,\,\, v(x,0)=b(x)\geq0,\,\,\,x\in\mathbb{R^{N}},$$
where $0<\alpha, \delta<1\leq\gamma,\beta\leq2.$

The critical exponent  of this problem were shown to be $$N\le  \max \biggl\{\frac{\frac{\delta}{q}+\alpha-\biggl(1-\frac{1}{pq}\biggr)}{\frac{\delta}{\gamma qp'}+\frac{\alpha}{\beta q'}},\frac{\frac{\alpha}{p}+\delta-\biggl(1-\frac{1}{pq}\biggr)}{\frac{\alpha}{\beta qp'}+\frac{\delta}{\gamma q'}}\biggr\}.$$

\subsection{Preliminaries}
\begin{definition}\cite{Kilbas} The left and right Riemann-Liouville fractional integrals $I_{a+,t}^{\alpha }$ and $I_{b-,t}^{\alpha }$ for an integrable function $f$ are given by
$$I^\alpha_{a+,t}f(t ) =\frac{1}{\Gamma \left( \alpha  \right)}\int\limits_{a}^{t}{{{\left( t-s \right)}^{\alpha -1}}}f\left( s \right)ds, t\in \left( a,b \right]$$
and
$$I^\alpha_{b-,t}f(t)=\frac{1}{\Gamma \left( \alpha  \right)}\int\limits_{t}^{b}{{{\left( s-t \right)}^{\alpha -1}}}f(s)ds, t\in \left[ a,b \right),$$
respectively. Here $0<\alpha \in \mathbb{R}$ and $\Gamma \left( \alpha  \right)$ denotes the Euler gamma function.\end{definition}

\begin{definition}\cite{Kilbas} The left Riemann-Liouville fractional derivative $D_{a+,t}^{\alpha }$  of order $0<\alpha <1$, for $f\in C^1([a,b])$ is defined by
$${D}_{a+,t}^{\alpha } f(t)=\frac{d}{dt}I_{{a+,t}}^{1-\alpha } f(t)=\frac{1}{\Gamma(1-\alpha)}\frac{d}{dt}\int\limits_{a}^{t}{{(t-s)^{-\alpha }}}{f}\left( s \right)ds,\,\,\, \forall t\in \left( a,b \right].$$

Similarly, the right Riemann-Liouville fractional derivative $D_{b-,t}^{\alpha }$  of order $0<\alpha <1$, for $f\in C^1([a,b])$ is defined by
$${D}_{b-,t}^{\alpha } f(t)=-\frac{d}{dt}I_{b-,t}^{1-\alpha } f(t)=-\frac{1}{\Gamma(1-\alpha)}\frac{d}{dt}\int\limits_{t}^{b}{{(s-t)^{-\alpha }}}{f}\left( s \right)ds,\,\,\, \forall t\in \left[ a,b \right).$$\end{definition}
\begin{definition}\cite{Kilbas} The $\alpha \in (0,1)$ order left and right Caputo fractional derivatives for $f\in C^1([a,b])$ are defined, respectively, by
$$\mathcal{D}_{a+,t}^{\alpha } f(t)=I_{a+,t}^{1-\alpha } \frac{d}{dt}f(t)=\frac{1}{\Gamma(1-\alpha)}\int\limits_{a}^{t}{{(t-s)^{-\alpha }}}{f'}\left( s \right)ds,\,\,\, \forall t\in \left( a,b \right]$$
and
$$\mathcal{D}_{b-,t}^{\alpha } f(t)=-I_{b-,t}^{1-\alpha }\frac{d}{dt} f(t)=-\frac{1}{\Gamma(1-\alpha)}\int\limits_{t}^{b}{{(s-t)^{-\alpha }}}{f'}\left( s \right)ds,\,\,\, \forall t\in \left[ a,b \right).$$

\end{definition}
\begin{definition}\label{D11}\cite{Kilbas} The Mittag-Leffler function is defined by
$$E_{\alpha,1}(z)=\sum_{k=0}^{\infty}\frac{z^k}{\Gamma(\alpha k+1)}, \alpha>0, z\in \mathbb{C}.$$
The two-parameters Mittag-Leffler function is defined by
$$E_{\alpha,\beta }(z)=\sum_{k=0}^{\infty}\frac{z^k}{\Gamma(\alpha k+\beta)},  z\in \mathbb{C},$$
where $\alpha>0$ and $\beta\in\mathbb{R}$ are arbitrary constants.
\end{definition}
\begin{lemma}\label{L1} \cite{Simon} For every $\alpha \in (0,1),$ the uniform bilateral estimate
$$\frac{1}{1+\Gamma(1-\alpha)x}\leq E_{\alpha,1}(-x)\leq\frac{1}{1+[\Gamma(1+\alpha)]^{-1}x}$$
holds over $\mathbb{R}^+$.
\end{lemma}
\begin{remark}\label{R1} Obviously, $0<E_{\alpha,1}(-x)<1,$ for any $x>0$ by Lemma \eqref{L1}.
\end{remark}
\begin{lemma}\label{L2}\cite{Hormander}  The Fourier transform of the n-dimensional $\delta(x)$ in ${\mathbb{R}^N}$ defined by
$$F\{\delta^N(x);\xi\}=\frac{1}{(2\pi)^{N}}\int\limits_{\mathbb{R}^N}{{e^{-i<x,\xi>}}\delta^N(x)dx}=1, \xi \in {\mathbb{R}^N}$$
and the inverse Fourier transform of $\delta(x)$ can be written as
$$\delta^N(x)=F^{-1}\{1\}=\frac{1}{(2\pi)^{N}}\int\limits_{\mathbb{R}^N}{{e^{i<x,\xi>}}d\xi}, \xi \in {\mathbb{R}^N},$$
where $<x,\xi>=\sum_{j=1}^N{x_j\xi_j}.$

The Dirac delta function $\delta^{(N)}(x)$ in the case of a $N$-dimensional domain ${\mathbb{R}^N}$, where $x = (x_1, x_2, ..., x_n)\in {\mathbb{R}^N}$ in \cite{{Salasnich}}:
\begin{equation*}
\delta^{(N)}(x)=
 \begin{cases}
   +\infty \,\,\, \text{for} \,\,\,  x=0,\\
   0 \,\,\, \text{for} \,\,\, x\neq0.
 \end{cases}
\end{equation*}
and
$$\int\limits_{\mathbb{R}^N}{{\delta^{(N)}(x)}dx}=1.$$

\end{lemma}
\begin{lemma}\label{L3}\cite{Hormander}(Hausdorff-Young inequality) Let $f\in L^1(\mathbb{R}^N)$ and $g \in L^p(\mathbb{R}^N), p\geq1.$ Then we have $h=f\ast g\in L^p(\mathbb{R}^N)$ and
$$\left\|h\right\|_{L^p(\mathbb{R}^N)}\leq\left\|f\right\|_{L^1(\mathbb{R}^N)}\cdot\left\|g\right\|_{L^p(\mathbb{R}^N)},$$
where $f\ast g=\int_{\mathbb{R}^N}f(x-y)g(y)dy.$
\end{lemma}
\begin{property}\label{P2.4}\cite{Kilbas} It holds
$$\int\limits_{\Omega_T}I_{0+,t}^{1-\alpha }u( x,t ){f}( x,t )dxdt=\int\limits_{\Omega_T}u( x,t ){{I}_{T-,t}^{1-\alpha} {f}}( x,t )dxdt.$$
\end{property}

\section{Linear integro-differential diffusion equation}
\subsection{Green function}
First of all, we consider the homogeneous initial value problem
\begin{equation}\label{2.1} {{u}_{t}}\left( x,t \right)-\Delta_x D_{0+,t}^{1-\alpha }u\left( x,t \right)=0,\text{ }0<\alpha <1,\text{ }x\in \mathbb{R}^N,\text{ }t>0,\end{equation}
\begin{equation}\label{2.2} u\left( x,0 \right)=u_0 \left( x \right),\text{ }x\in \mathbb{R}^N.\end{equation}
\begin{theorem}\label{T2.1}The solution of the homogeneous integro-differential problem \eqref{2.1}-\eqref{2.2} can be represented by
\begin{equation}\label{2.3}u\left( x,t \right)=\int\limits_{\mathbb{R}^{N}}{G\left( x-y,t \right)u_0 \left( y \right)dy,}\end{equation} where $G\left( x,t \right)=\frac{1}{{(2\pi)^{N}}}\int\limits_{\mathbb{R}^N}{{{e}^{-i<x,\xi>}}{{E}_{\alpha ,1}}\left( -{{\xi }^{2}}{{t}^{\alpha }} \right)d\xi }.$\end{theorem}
\begin{proof} Using the Fourier transform to problem \eqref{2.1}-\eqref{2.2} with respect to variable $x$ yields
\begin{equation}\label{2.4}\widetilde{{{u}_{t}}}\left( \xi ,t \right)+{{\xi }^{2}}D_{0+,t}^{1-\alpha }\widetilde{u}\left( \xi ,t \right)=0,\end{equation}
\begin{equation}\label{2.5}\widetilde{u}\left( \xi ,0 \right)=\widetilde{u}_0\left( \xi  \right).\end{equation}

We can easily prove that
\begin{equation}\label{2.6}\widetilde{u}\left( \xi ,t \right)={{E}_{\alpha ,1}}\left( -{{\xi }^{2}}{{t}^{\alpha }} \right)\widetilde{u}_0 \left( \xi  \right)\end{equation}
is the solution of the equations  \eqref{2.4}-\eqref{2.5}.

After that, by using the inverse Fourier transform to equation \eqref{2.6}, we have
\begin{multline}\label{2.7}
u\left( x,t \right)=\frac{1}{{(2\pi)^{N}}}\int_{\mathbb{R}^N}{e}^{-i<x,\xi>}{{}{{E}_{\alpha ,1}}\left( -{{\xi }^{2}}{{t}^{\alpha }} \right)\widetilde{u}_0 \left( \xi \right)d\xi }
\\=\frac{1}{{(2\pi)^{N}}}\int_{\mathbb{R}^N}{e}^{-i<x,\xi>}{{}{{E}_{\alpha ,1}}\left( -{{\xi }^{2}}{{t}^{\alpha }} \right)\int_{\mathbb{R}^N}{{{e}^{i<y,\xi>}}u_0 \left( y \right)dy}d\xi }
\\=\frac{1}{{(2\pi)^{N}}}\int_{\mathbb{R}^N}{\int_{\mathbb{R}^N}{{{e}^{-i<x-y,\xi>}}{{E}_{\alpha ,1}}\left( -{{\xi }^{2}}{{t}^{\alpha }} \right)}d\xi u_0 \left( y \right)dy}\\
=\int_{\mathbb{R}^N}{G\left( x-y,t \right)u_0 \left( y \right)dy,}\end{multline}
where  $G\left( x,t \right)=\frac{1}{(2\pi)^{N}}\int_{\mathbb{R}^N}{{{e}^{-i<x,\xi>}}{{E}_{\alpha ,1}}\left( -{{\xi }^{2}}{{t}^{\alpha }} \right)d\xi }.$ \end{proof}
\begin{lemma}\label{L2.2}The Green function $G(x,t)$ of the problem \eqref{2.1}-\eqref{2.2} has the following estimate:
\begin{equation}\label{GF1}\int\limits_{\mathbb{R}^N}G(x,t)dx<1,\, t>0.\end{equation}
\end{lemma}
\begin{proof}
Accordingly to Lemma \ref{L1}, we have that
\begin{align*}
G\left( x,t \right)&=\frac{1}{(2\pi)^{N}}\int_{\mathbb{R}^N}{{{e}^{-i<x,\xi>}}{{E}_{\alpha ,1}}\left( -{{\xi }^{2}}{{t}^{\alpha }} \right)d\xi }
\\& \leq\frac{1}{(2\pi)^{N}}\left|\int_{\mathbb{R}^N}{{{e}^{-i<x,\xi>}}{{E}_{\alpha ,1}}\left( -{{\xi }^{2}}{{t}^{\alpha }} \right)d\xi }\right|
\\& \leq\frac{1}{(2\pi)^{N}}\int_{\mathbb{R}^N}{{{e}^{-i<x,\xi>}}}\left|{{{E}_{\alpha ,1}}\left( -{{\xi }^{2}}{{t}^{\alpha }} \right) }\right|d\xi
\\& <\frac{1}{(2\pi)^{N}}\int_{\mathbb{R}^N}{{{e}^{-i<x,\xi>}}}\cdot 1d\xi\\&=\delta^N(x).\end{align*}
From Lemma \ref{L2} we obtain
$$\int\limits_{\mathbb{R}^N}{{G(x,t)}dx}<\int\limits_{\mathbb{R}^N}{{\delta^N(x)}dx}=1,\,\,\,t>0,$$
which completes the proof.
\end{proof}

\subsection{Duhamel principle for the integro-differential diffusion equation}

Now, we consider the nonhomogeneous initial value problem
\begin{equation}\label{2.8}{{u}_{t}}\left( x,t \right)-\Delta_xD_{0+,t}^{1-\alpha }u\left( x,t \right)=f\left( x,t \right),\text{ }0<\alpha <1,\text{ }x\in \mathbb{R}^N,\text{ }t>0,\end{equation}
\begin{equation}\label{2.9}u\left( x,0 \right)=0,\text{ }x\in \mathbb{R}^N.\end{equation}
 \begin{theorem}\label{T2.3}(Duhamel principle) The solution of the integro-differential diffusion problem \eqref{2.8}-\eqref{2.9} has the form
\begin{equation}\label{2.10}u\left( x,t \right)=\int_{0}^{t}{w\left( x,t;\tau  \right)d\tau ,}\end{equation}
where $w\left( x,t;\tau  \right)$ is the solution of homogeneous equation
\begin{equation}\label{2.11}{{w}_{t}}\left( x,t;\tau \right)-\Delta_xD_{\tau+,t}^{1-\alpha }w\left( x,t;\tau \right)=0,\text{ }x\in \mathbb{R}^N,\text{ }t>\tau ,\end{equation}
satisfying
\begin{equation}\label{2.12}\tau= t :w\left( x,t;\tau  \right)=f\left( x,t \right),\end{equation}
where $f\left( x,t \right)$ is the sufficiently smooth function.\end{theorem}
\begin{proof} Suppose that $w\left( x,t;\tau  \right)$ is the solution of the integro-differential diffusion problem \eqref{2.11}-\eqref{2.12}. We will prove that $u\left( x,t \right)=\int_{0}^{t}{w\left( x,t;\tau  \right)d\tau }$ is the solution of the problem \eqref{2.8}-\eqref{2.9}.

According  to the results of Umarov in \cite{Umarov2}, we have
\begin{equation}\label{2.13}\frac{\partial }{\partial t}u\left( x,t \right)={{\left. w\left( x,t;\tau  \right) \right|}_{\tau =t}}+\int_{0}^{t}{\frac{\partial }{\partial t}w\left( x,t;\tau  \right)d\tau .}\end{equation}
Now, using the Fubini's theorem we calculate the fractional derivative of the function \eqref{2.10}
\begin{multline}\label{2.14}D_{0+,t}^{1-\alpha }u\left( x,t \right)=D_{0+,t}^{1-\alpha }\int_{0}^{t}{{{w}}\left( x,t;\tau  \right)d\tau }\\=\frac{1}{\Gamma \left( \alpha  \right)}\frac{\partial}{\partial t}\int_{0}^{t}{{{\left( t-s \right)}^{\alpha -1}}\int_{0}^{s}{{{w}}\left( x,s;\tau  \right)d\tau }ds}
\\=\frac{1}{\Gamma \left( \alpha  \right)}\frac{\partial}{\partial t}\int_{0}^{t}{\int_{\tau }^{t}{{{w}}\left( x,s;\tau  \right){{\left( t-s \right)}^{\alpha -1}}dsd\tau }}\\= \frac{1}{\Gamma \left( \alpha  \right)}\int_{0}^{t}\frac{\partial}{\partial t}{\int_{\tau }^{t}{{{w}}\left( x,s;\tau  \right){{\left( t-s \right)}^{\alpha -1}}dsd\tau }}\\=\int_{0}^{t}D_{\tau+,t}^{1-\alpha }w(x,t;\tau)d\tau.\end{multline}
Consequently
\begin{equation}\label{2.15}\Delta_xD_{0+,t}^{1-\alpha }u\left( x,t \right)=\int_{0}^{t}\Delta_xD_{\tau+,t}^{1-\alpha }w(x,t;\tau)d\tau.\end{equation}

According to the initial data, the second term of the right side of the equation is zero.

Therefore, using \eqref{2.13} and \eqref{2.15}, the equation \eqref{2.8} has the form
\begin{align*}{{u}_{t}}\left( x,t \right)&-\Delta_xD_{0+,t}^{1-\alpha }u\left( x,t \right)\\&={{\left. w\left( x,t;\tau  \right) \right|}_{\tau =t}}+\int_{0}^{t}{\frac{\partial }{\partial t}w\left( x,t;\tau  \right)d\tau -}\Delta_xD_{0+,t}^{1-\alpha }\int_{0}^{t}{{{w}}\left( x,t;\tau  \right)d\tau }\\&=
{{\left. w\left( x,t;\tau  \right) \right|}_{\tau =t}}+\int_{0}^{t}{\left[ \frac{\partial }{\partial t}w\left( x,t;\tau  \right)-\Delta_xD_{\tau+,t}^{1-\alpha }{{w}}\left( x,t;\tau  \right) \right]d\tau }=f\left( x,t \right).\end{align*}
\end{proof}
\begin{corollary}\label{C1}
(i) The initial value problem \eqref{2.11}-\eqref{2.12} admits a solution. Let ${t}'=t-\tau $ in \eqref{2.11}-\eqref{2.12}, then the integro-differential diffusion problem \eqref{2.11}-\eqref{2.12} can be written in the following form
\begin{equation}\label{2.16}{{w}_{t}}\left( x,{t}';\tau  \right)-\Delta_xD_{\tau+,t}^{1-\alpha }w\left( x,{t}';\tau  \right)=0,\text{ }x\in \mathbb{R}^N,\text{ }{t}>\tau ,\end{equation}
\begin{equation}\label{2.17}t=\tau :w\left( x,0;\tau  \right)=f\left( x,\tau \right),\text{ }x\in \mathbb{R}^N.\end{equation}
\end{corollary}

Theorem \eqref{T2.1} implies that the solution of the problem \eqref{2.16}-\eqref{2.17} can be obtained by
\begin{equation}\label{2.18}w\left( x,{t}';\tau  \right)=\int_{\mathbb{R}^N}{G\left( x-y,{t}' \right)}f\left( y,\tau  \right)dy.\end{equation}

Hence, the solution of the initial value problem \eqref{2.11}-\eqref{2.12} can be represented as
\begin{equation}\label{2.19}w\left( x,t-\tau ;\tau  \right)=\int_{\mathbb{R}^N}{G\left( x-y,t-\tau  \right)}f\left( y,\tau  \right)dy.\end{equation}
\emph{(ii)} Furthermore, by Theorem \eqref{T2.3}, the solution of the initial value problem \eqref{2.8}-\eqref{2.9} has the form
\begin{equation}\label{2.20}u\left( x,t \right)=\int_{0}^{t}{w\left( x,t-\tau ;\tau  \right)d\tau }=\int_{0}^{t}{\int_{\mathbb{R}^N}{G\left( x-y,t-\tau  \right)}f\left( y,\tau  \right)dyd\tau }.\end{equation}

Combining Theorem \ref{T2.1} with Corollary \ref{C1}, we can get the following Theorem \ref{T2.5}.
\begin{theorem}\label{T2.5}The solution of the nonhomogeneous initial value problem \eqref{2.1}-\eqref{2.2} has the form
\begin{equation}\label{2.22}u\left( x,t \right)=\int_{\mathbb{R}^N}{G\left( x-y,t \right)}u_0 \left( y \right)dy+\int_{0}^{t}{\int_{\mathbb{R}^N}{G\left( x-y,t-\tau  \right)}f\left( y,\tau  \right)dyd\tau }.\end{equation}
\end{theorem}
\subsection{Stability of solution of the integro-differential diffusion equation}
In this subsection we study a stability of solution of \eqref{2.1} with initial condition \eqref{2.2}.
\begin{theorem}\label{T2.6} Let $t\rightarrow 0,$ then the solution \eqref{2.22} of the initial problem \eqref{1.1}-\eqref{1.2} is satisfying the following estimate
 $$\|u(x,t)\|_{L^p({\mathbb{R}^N})}\leq\|u_0(x)\|_{L^p({\mathbb{R}^N})}, \,\,\,\,\ x \in {\mathbb{R}^N}, p\geq 1.$$\end{theorem}
Proof. From \eqref{2.22} and Lemma \eqref{L2}, we have
 \begin{multline}\label{2.23}
\lim_{t\rightarrow0}{\|u(x,t)\|_{L^p({\mathbb{R}^N})}}=\left\|\int_{\mathbb{R}^N}{G(x-y,0)u(y,0)dy}\right\|_{L^p({\mathbb{R}^N})}\\
 =\left\|\int_{\mathbb{R}^N}{\delta(x-y)u_0(y)dy}\right\|_{L^p({\mathbb{R}^N})} =\left\|\delta(x)\ast u_0(x)\right\|_{L^p({\mathbb{R}^N})}, p\geq 1.\end{multline}

 Then according to the inequality \eqref{2.23}, Lemma \eqref{L3} and Lemma \eqref{L2} we imply
 \begin{align*}
\lim_{t\rightarrow0}{\|u(x,t)\|_{L^p({\mathbb{R}^N})}}&=\left\|\delta(x)\ast u_0(x)\right\|_{L^p({\mathbb{R}^N})}\\&\leq\left\|\delta(x)\right\|\cdot\left\|u_0(x)\right\|_{L^p({\mathbb{R}^N})}\\ &\leq\left\|u_0(x)\right\|_{L^p({\mathbb{R}^N})}, \end{align*}
 for $p\geq 1,$ which completes the proof.

 \begin{theorem}\label{T2.7} (Stability) Let $u_{0}(x)\in L^p({\mathbb{R}^N}),p\geq 1.$ Then the solution $u(x,t)$ of the inhomogeneous problem \eqref{2.1}-\eqref{2.2} is stable.\end{theorem}
 \begin{proof} Suppose that $u_1(x,t)$ is the solution of the inhomogeneous problem \eqref{2.1}
\begin{equation*} {{u}_{t}}\left( x,t \right)-\Delta_x D_{0+,t}^{1-\alpha }u\left( x,t \right)=0,\text{ }x\in \mathbb{R}^N,\text{ }t>0,\end{equation*}
\begin{equation}\label{26}u\left( x,0 \right)={{u}_{0}^{(1)}}\left( x \right), \,x\in {\mathbb{R}^N}\end{equation}
and $u_2(x,t)$ is the solution of  \eqref{2.1} with
\begin{equation}\label{28}u\left( x,0 \right)={{u}_{0}^{(2)}}\left( x \right), \,x\in {\mathbb{R}^N}.\end{equation}

Then $u_1(x,t)-u_2(x,t)$ is the solution of the following problem
\begin{equation}\label{29}{{u}_{t}}(x,t)-\Delta_x {D}_{0+,t}^{1-\alpha }u(x,t)=0,\text{ }(x,t)\in \mathbb{R^N} \times \left( 0,T \right)=\Omega_T, \end{equation}
\begin{equation}\label{30}u\left( x,0 \right)={{u}_{0}^{(1)}(x)}-{{u}_{0}^{(2)}}\left( x \right), \,x\in {\mathbb{R}^N}.\end{equation}

According to Theorem \ref{T2.1}, we have
\begin{equation}\label{31}u_1(x,t)-u_2(x,t)=\int_{\mathbb{R}^N}{G(x-y,t)\left[u_{0}^{(1)}(y)-u_{0}^{(2)}(y)\right]dy.}\end{equation}
Hence, by taking the $L^p$-norm on both sides of the expression \eqref{31} and applying Lemma \ref{L3} it yields
\begin{align*}
\left\|u_1(x,t)-u_2(x,t)\right\|_{L^p({\mathbb{R}^N})}&=\left\|\int_{\mathbb{R}^N}{G(x-y,t)\left[u_{0}^{(1)}(y)-u_{0}^{(2)}(y)\right]dy}\right\|_{L^p({\mathbb{R}^N})}\\
&=\left\|G(x,t)\ast\left[u_{0}^{(1)}(x)-u_{0}^{(2)}(x)\right]\right\|_{L^p({\mathbb{R}^N})}\\ &\leq\left\|G(x,t)\right\|_{L^1({\mathbb{R}^N})}\cdot\left\|u_{0}^{(1)}(x)-u_{0}^{(2)}(x)\right\|_{L^p({\mathbb{R}^N})}, \end{align*}
for $t>0$.

According to Lemma \ref{L3} and Lemma \ref{L2.2}, it follows
\begin{align*}
\left\|u_1(x,t)-u_2(x,t)\right\|_{L^p({\mathbb{R}^N})}&=\left\|G(x,t)\right\|_{L^1(\mathbb{R}^N)}\cdot\left\|u_{0}^{(1)}(x)-u_{0}^{(2)}(x)\right\|_{L^p({\mathbb{R}^N})}\\
&<\left\|u_{0}^{(1)}(x)-u_{0}^{(2)}(x)\right\|_{L^p({\mathbb{R}^N})},\,t>0. \end{align*}

For any $\varepsilon>0$, choose $\delta<\varepsilon$. Then $\left\|u_{0}^{(1)}(x)-u_{0}^{(2)}(x)\right\|_{L^p({\mathbb{R}^N})}<\delta$ implies $$\left\|u_1(x,t)-u_2(x,t)\right\|_{L^p({\mathbb{R}^N})}<\varepsilon,\,t>0.$$

Therefore the solution $u(x,t)$ of the inhomogeneous inital-value problem \eqref{2.1}-\eqref{2.2} is stable.
 \end{proof}

\section{Nonlinear integro-differential diffusion equation}
\subsection{Local mild solution of the integro-differential diffusion equation}
In this section we study for local mild solution of \eqref{1.1} with initial condition \eqref{1.2}.

\begin{definition}\label{D2.1}
A function $u\in L_{loc}^{1}(\Omega_T)$, $(\Omega_T:=(x,t)\in \mathbb{R}^N \times \left( 0,T \right))$  is a local weak solution to time-fractional diffusion equation \eqref{1.1} on $\Omega_T$ if
\begin{multline}\label{4.8}-\int\limits_{\mathbb{R}^N}{u_0( x)\varphi (x,0)dx}-\int\limits_{\Omega_T}{u\left( x,t \right){{{\varphi} }_{t}}\left( x,t \right)dx dt}\\=\int\limits_{\Omega_T}{u\left( x,t \right){\mathcal{D}}_{T-,t}^{1-\alpha }{\Delta_x{\varphi}}(x,t)dxdt} +\int\limits_{\Omega_T}{f( x,t,u ){\varphi} ( x,t)dxdt}\end{multline}
for any test function $\varphi(x,t)\in C_{x,t}^{2,1}( \Omega_{T})$ with $\varphi( x,T )=0.$
\end{definition}

\begin{definition}\label{D2.7} Let $u_0(x)\in L^{\infty}(\mathbb{R}^N))$ and $T>0$. We say $u\in C_0(\mathbb{R}^N,C(0,T_{\max}))$ is a mild solution of \eqref{1.1}-\eqref{1.2} if $u$ satisfies the following integral equation
$$u\left( x,t \right)=\int_{\mathbb{R}^N}{G\left( x-y,t \right)}u_0(y)dy+\int_{\Omega_T}{{G\left( x-y,t-\tau  \right)}f(y,\tau,u)dyd\tau }.$$
\end{definition}

\begin{definition} Assume $(X,d)$ is a metric space. We say that $T: X\rightarrow X$ is contraction mapping on $X,$ if there exists $M \in (0, 1)$ such that $$d(T(x),T(y))\leq Md(x,y),$$ for all $x,y\in X.$ \end{definition}
\begin{theorem} Let $(X, d)$ be a non-empty complete metric space with a contraction mapping $T : X \rightarrow X.$ Then $T$ admits a unique fixed-point $x^*$ in $X$ (i.e. $T(x^*) = x^*$). \end{theorem}
\begin{theorem}\label{T4.2}(Local existence). Let $u_0(x) \in C_0(\mathbb{R}^N)$ and $f(x,t,u)$ be a local Lipschitz function. Then there exists a unique mild solution $u(x,t)\in C_0(\mathbb{R}^N,C[0, T_{\max}))$ of problem \eqref{1.1}-\eqref{1.2} , where $T_{\max} > 0$ is the maximal time of existence. \end{theorem}
Proof.
For arbitrary $T >0$, we define the Banach space
$$E_T=\{u(x,t)\in C_0(\mathbb{R}^N,C[0, T_{\max})); |||u(x,t)|||\leq2\parallel u_0(x)\parallel_{L^{\infty}(\mathbb{R}^N)}\}.$$
and $|||\cdot|||$ is the norm of $E_T$ defined by
$$|||u(x,t)||| = \parallel u(x,t)\parallel_{L^{\infty}(\mathbb{R}^N,L^{\infty}[0,T])} .$$

Next, for $u(x,t)\in E_T$, we define
$$\Psi(u)=\int_{\mathbb{R}^N}{G\left( x-y,{t}' \right)}u_0(y)dy+\int_{\Omega_T}{{G\left( x-y,t-\tau  \right)}f(y,\tau,u)dyd\tau }.$$
We are going to prove the existence of a unique local solution as a fixed point of $\Psi$ via the Banach fixed point theorem.
\begin{itemize}
  \item $\Psi:E_T\rightarrow E_T.$
\end{itemize}
Let $u(x,t)\in E_T$, using \eqref{2.22}, we obtain
\begin{align*}|||u(x,t)|||&\leq \parallel u_0(x) \parallel_{L^{\infty}(\mathbb{R}^N)}+\left\|\int_{\Omega_T}{{G\left( x-y,t-\tau  \right)}f(y,\tau,u)dyd\tau }\right\|_{L^{\infty}[0,T]}\\&=\parallel u_0(x) \parallel_{L^{\infty}(\mathbb{R}^N)}+\left\|\int_{\Omega_T}{{G\left( x-y,t-\tau  \right)}\left\|f(y,\tau,u)\right\|_{L^{\infty}(\mathbb{R}^N)} dyd\tau }\right\|_{L^{\infty}[0,T]}\\&\leq \parallel u_0(x) \parallel_{L^{\infty}(\mathbb{R}^N)}+T\left\|u(x,t)\right\|_{L^{\infty}(\mathbb{R}^N,L^{\infty}[0,T])}\\&\leq \parallel u_0(x) \parallel_{L^{\infty}(\mathbb{R}^N)}+2T\left\|u_0(x)\right\|_{L^{\infty}(\mathbb{R}^N)}. \end{align*}

 Now, for $T$ such that \, $2T\left\|u_0(x)\right\|_{L^{\infty}(\mathbb{R}^N)}\leq \parallel u_0(x) \parallel_{L^{\infty}(\mathbb{R}^N)},\,\, u(x,t)\in E_T$ and have chosen $T$ so that $T\leq\frac{1}{2}.$

$\bullet$ $\Psi$ is a contraction for $T\leq\frac{1}{2}.$

Let $u(x,t),\tilde{u}(x,t) \in E_T;$ using Lemma \eqref{L2.2}, we get
\begin{align*}|||u(x,t)-\tilde{u}(x,t)|||&\leq\left\|\int_{\Omega_T}{{G\left( x-y,t-\tau  \right)}\left\|f(y,\tau,u)-f(y,\tau,\tilde{u})\right\|_{L^{\infty}(\mathbb{R}^N)} dyd\tau }\right\|_{L^{\infty}[0,T]}\\&\leq T\left\|f(y,\tau,u)-f(y,\tau,\tilde{u})\right\|_{L^{\infty}(\mathbb{R}^N,L^{\infty}[0,T])}
\\&\leq T\left\|u-\tilde{u}\right\|_{L^{\infty}(\mathbb{R}^N,L^{\infty}[0,T])}
\\&\leq \frac{1}{2}|||u(x,t)-\tilde{u}(x,t)|||.
\end{align*}
and have chosen $T$ so that $T\leq\frac{1}{2}.$

Hence, by the Banach fixed point theorem, the problem \eqref{1.1}-\eqref{1.2} admits a mild solution $u(x,t)\in E_T$.

The Banach fixed point theorem then ensures the existence of a mild solution of problem \eqref{1.1}-\eqref{1.2}.
\begin{theorem} Let $u$ be a local solution of problem \eqref{1.1}-\eqref{1.2} for $T<+\infty$. Then we obtain the following estimate
$$\inf_{|x|\rightarrow \infty}[u_0(x)h^{p'/p}(x)]\leq C T^{p'+1}$$
where $C \in \mathbb{R}^+$  .
\end{theorem}
Proof. We follow the idea of Baras and Kersners \cite{Baras}. Let us consider the following test function:
\begin{equation*}
\varphi(x,t)=\Phi\biggl(\frac{x}{R}\biggr)
 \begin{cases}
   \biggl(1-\frac{t}{T}\biggr)^l, \,\,\,\ 0<t\leq T\\
   0, \,\,\,\ t>T
 \end{cases}
\end{equation*}
where $\Phi\in W^{1,\infty}(R^N)$ is nonnegative with $\textrm{supp}\,\Phi\subset\{{1<|x|<2}\}$ and satisfies
$\Delta_x\Phi(x)\leq k\Phi(x)$ for some $k>0$.

The exponent $l$ is any positive real number if $p>1/(1-\alpha)$ and $l>\alpha p'-1$ if $p<1/(1-\alpha)$. We have
$$\mathcal{D}_{T-,t}^{1-\alpha }\Delta_x\varphi(x,t)=\mathcal{D}_{T-,t}^{1-\alpha }\Delta_x\biggl(1-\frac{t}{T}\biggr)^l=\Theta \Delta_x\Phi(x) R^{-2}{T}^{1-\alpha}\biggl(1-\frac{t}{T}\biggr)^{l+\alpha-1},$$
where $\Theta=\frac{\Gamma(l)}{\Gamma(l+\alpha)}$.

Using the formulation of Definition \eqref{D2.1} and a similar argument to the one which lead us to \eqref{2.30} but
keeping the first term in the left-hand side of \eqref{D2.1}, we have

\begin{multline}\label{4.23}-\int\limits_{\mathbb{R}^N}{u_0( x)\varphi (x,0)dx}\leq\\ C(\varepsilon) \int\limits_{{\Omega}_{T}}{\left\{|\mathcal{D}_{T-,t}^{1-\alpha }{\Delta_x{\varphi }(x,t)}|_+^{p'}+{|\varphi_{t}(x,t)\arrowvert}_+^{p'}\right\}(h{\varphi })^{1-p'}(x,t)}dxdt\end{multline}
where $C \in \mathbb{R}^+$ . In view of the hypotheses on $l$ and the point is that

$$\mathcal{D}_{T-,t}^{1-\alpha }\Delta_x\varphi(x,t)=\Lambda \Delta_x\Phi(x) R^{-2}{T}^{1-\alpha}\biggl(1-\frac{t}{T}\biggr)^{l+\alpha-1}.$$

Let us perform the change of variables  $t=T\tau$ and $x=Ry$ in \eqref{4.23}, we have
\begin{equation}\label{4.24}\int\limits_{\mathbb{R^N}}{u_0(Ry)\Phi(y)}\leq C R^{-2p'}T^{(1-\alpha)p'+1}\int\limits_{{\Omega}_{T}}\Phi(y){h^{1-p'}(Ry)}+CT^{p'+1}\int\limits_{{{\Omega}_{T}}}\Phi(y){h^{1-p'}(Ry)}.\end{equation}

Using the estimate

$$\inf_{|y|>1}u_0(Ry){h^{p'-1}(Ry)}\int\limits_{\mathbb{R}^N}{\Phi(y){h(Ry)}^{1-p'}}\leq\int\limits_{\mathbb{R}^N}{u_0(Ry)\Phi(y)}$$
and dividing by the term $\int\limits_{\mathbb{R}^N}{u_0(Ry)\Phi(y)}$, we can write
$$\inf_{|y|>1}u_0(Ry){h^{p'-1}(Ry)}\leq C (R^{-2p'}T^{(1-\alpha)p'+1}+T^{p'+1}).$$
Let $R\rightarrow +\infty$, then we have
$$\inf_{|x|\rightarrow \infty}[u_0(x)h^{p'-1}(x)]\leq C T^{p'+1}.$$

\subsection{Critical exponents of Fujita type for the integro-differential diffusion equation}
In this section we study the following integro-differential diffusion equation
\begin{equation}\label{F.1}{{u}_{t}}(x,t)=\Delta_x {D}_{0+,t}^{1-\alpha }u(x,t)+t^\sigma|x|^\rho u^p(x,t),\text{ }(x,t)\in \mathbb{R^N} \times \left( 0,T \right)=\Omega_T,\end{equation}
with the initial condition
\begin{equation}\label{F.2}u\left( x,0 \right)={{u}_{0}}\left( x \right)\ge 0,\end{equation}
where $\sigma,\rho \in \mathbb{R}$ and $\alpha \in (0,1)$.
\begin{definition}\label{WS} A function $u\in L_{loc}^{1}(\Omega_\infty)$, $(\Omega_\infty:=(x,t)\in \mathbb{R}^N \times \left( 0,\infty \right))$  is a global weak solution to integro-differential diffusion equation on $\Omega_\infty$ if
\begin{multline}\label{2.25}-\int\limits_{\mathbb{R}^N}{u_0( x)\varphi (x,0)dx}-\int\limits_{\Omega_\infty}{|x|^\rho t^\sigma u^{p}( x,t ){\varphi} ( x,t)dxdt}=\\ \int\limits_{\Omega_\infty}{u\left( x,t \right){\mathcal{D}}_{T-,t}^{1-\alpha }{\Delta_x{\varphi}}(x,t)dxdt} +\int\limits_{\Omega_\infty}{u\left( x,t \right){{{\varphi} }_{t}}\left( x,t \right)dx dt,}\end{multline}
where $\rho \geq 0, \sigma>-1.$
\end{definition}

\begin{theorem}\label{T1} Let $p>1$. If
$$1<p\leq p_c=1+\frac{2(\sigma+1)+\rho\alpha}{N\alpha },$$
then problem \eqref{F.1}-\eqref{F.2} admits no global weak nonnegative solutions other than the trivial one.
\end{theorem}
\textbf{Proof.}
We prove from the contradiction. Assume that $u$ is a nontrivial nonnegative solution which exists globally in time. That is $u$ exists in $(0,T^*)$ for $T^*>0$.
Let $T, R, \theta \in \mathbb{R}^+$, such that $0<T{R}^{2/\theta}<{T}^{*}.$

Suppose $\Phi(z)$ be a smooth nonincreasing function
\begin{equation*}
\Phi(z)=
 \begin{cases}
   1, &\text{if $z\leq1 $}\\
   0, &\text{if $z\geq 2$}
 \end{cases}
\end{equation*}
and $0\leq \Phi(z)\leq1$.

The test function $\varphi(x,t)$ is chosen so that
\begin{multline}\label{2.29}\int\limits_{\Omega_{T{{R}^{2/\theta }}}}|\mathcal{D}_{\small T{{R}^{2/\theta }-,t}}^{1-\alpha }{\Delta_x{\varphi }(x,t)}|^{p'}({h\varphi})^{-p'/p}(x,t)dxdt<\infty, \\ \int\limits_{\Omega_{T{{R}^{2/\theta }}}}{{{\left| {{\varphi }_{t}(x,t)} \right|}^{p'}}}{({h\varphi})^{-p'/p}(x,t)}dxdt<\infty.\end{multline}

To estimate the right-hand side of the Definition \eqref{WS} on ${{\Omega }_{T{{R}^{2/\theta }}}}$, we write

\begin{align*}\int\limits_{{{\Omega }_{T{{R}^{2/\theta }}}}}&{u\left( x,t \right) \mathcal{D}_{T{{R}^{2/\theta }-,t}}^{1-\alpha }{\Delta_x{\varphi }}\left( x,t \right)dxdt}\\&= \int\limits_{{{\Omega }_{T{{R}^{2/\theta }}}}}{u\left( x,t \right)({h{\varphi })^{1/p}}\left( x,t \right)\left[ \mathcal{D}_{T{{R}^{2/\theta }-,t}}^{1-\alpha }{\Delta_x{\varphi }}\left( x,t \right) \right]{(h{\varphi })^{-1/p}}\left( x,t \right)dxdt}.\end{align*}

According to the $\varepsilon$-Young inequality
$$XY\leq \varepsilon X^p+C(\varepsilon)Y^{p'},\,\, \frac{1}{p}+\frac{1}{p'}=1,\,\, X\geq0,Y\geq0 $$
for the right-side of \eqref{2.25} on $\Omega_{T{R}^{2/\theta}}$, we have
\begin{multline*}\int\limits_{{\Omega}_{T{{R}^{2/\theta }}}}u(x,t)\mathcal{D}_{\small T{{R}^{2/\theta }}-,t}^{1-\alpha }{\Delta_x{\varphi }(x,t)}dxdt \\\leq \varepsilon\int\limits_{{\Omega}_{T{{R}^{2/\theta }}}}{\mid u\mid}^{p}(x,t)(h\varphi)(x,t) dxdt+C(\varepsilon) \int\limits_{{\Omega}_{T{{R}^{2/\theta }}}}{|\mathcal{D}_{T{{R}^{2/\theta }}-,t}^{1-\alpha }{\Delta_x{\varphi }(x,t)}|}^{p'}({h\varphi })^{-p'/p}(x,t)dxdt.\end{multline*}

Similarly,
\begin{multline*}\int\limits_{{{\Omega}_{T{{R}^{2/\theta }}}}}u(x,t){\varphi }_{t}(x,t) dxdt\\ \le \varepsilon\int\limits_{{\Omega}_{T{{R}^{2/\theta }}}}{\mid u(x,t)\mid}^{p}(h\varphi)(x,t) dxdt+C(\varepsilon) \int\limits_{{\Omega}_{T{{R}^{2/\theta }}}}{|\varphi_{t}(x,t)|}^{p'}(h{\varphi })^{-p'/p}(x,t)dxdt.\end{multline*}

Now, taking $\varepsilon$ small enough, we obtain the estimate
\begin{multline}\label{2.30}\int\limits_{{\Omega}_{T{{R}^{2/\theta }}}}h(x,t){u^{p}(x,t)}\varphi(x,t) dxdt \\ \leq C(\varepsilon) \int\limits_{{\Omega}_{T{{R}^{2/\theta }}}}{\left\{|\mathcal{D}_{T{{R}^{2/\theta }}-,t}^{1-\alpha }{\Delta_x{\varphi }(x,t)}|^{p'}+{|\varphi_{t}(x,t)\arrowvert}^{p'}\right\}(h{\varphi })^{-p'/p}(x,t)}dxdt.\end{multline}

We set the function $\varphi$ such that     $$\varphi \left( x,t \right)=\Phi \left( \frac{{{|x|}^{2}}+{{t}^{\theta }}}{{{R}^{2}}} \right),$$ where $R,\theta \in {\mathbb{{R}}^{+}}.$

Let us perform the change of variables  $t=\tau {{R}^{2/\theta }},$ $x=yR$ and set
$$\Omega :=\left\{ \left. \left( y,\tau  \right)\in \mathbb{R^N}\times (0,T/{R}^{2/\theta}),\text{ }{{y}^{2}}+{{\tau }^{\theta }}<2 \right\} \right.,\mu \left( y,\tau  \right)={{y}^{2}}+{{\tau }^{\theta }}.$$

Then, we choose $\theta$ such that the right-hand side of \eqref{2.30}
\begin{multline*} \int\limits_{{\Omega}_{T{{R}^{2/\theta }}}}|\mathcal{D}_{T{{R}^{2/\theta }}-,t}^{1-\alpha }{\Delta_x{\varphi }(x,t)}|^{p'}(h{\varphi })^{-p'/p}(x,t)dxdt\\=
\int\limits_{{\Omega}_{T{{R}^{2/\theta }}}}{\left[ -\frac{1}{\Gamma \left( \alpha  \right)}\int\limits_{\tau {{R}^{2/\theta }}}^{T{{R}^{2/\theta }}}{{(s-t)^{\alpha -1 }}}{\Delta_x{\varphi }_{s}}\left( x,s \right)ds \right]}^{p'}{(h{\varphi })^{-p'/p}}(x,t)dxdt\\ \leq R^{\frac{2}{\theta}(\alpha-1)p'-2p'-(\frac{2}{\theta}\sigma+\rho)\frac{p'}{p}+\frac{2}{\theta}+N}\int\limits_{\Omega}|\mathcal{D}_{T-,\tau}^{1-\alpha }({\Delta_y{\Phi }_{\tau}}\circ\mu)|^{p'}h^{-p'/p}(y,\tau)(\Phi \circ\mu)^{-p'/p}dyd\tau\end{multline*}
and
\begin{multline*}\int\limits_{{\Omega}_{T{{R}^{2/\theta }}}}{{{\left| {{\varphi }_{t}}(x,t) \right|}^{p'}}}{(h{\varphi })^{-p'/p}}(x,t)dxdt\\ \le{{R}^{-\frac{2}{\theta}p'-(\frac{2}{\theta}\sigma+\rho)\frac{p'}{p}+\frac{2}{\theta}+N}}\int\limits_{\Omega}{{{\left| ({{\Phi }_{\tau}}\circ \mu )\right|}^{p'}}}h^{-p'/p}(y,\tau){{\left( \Phi \circ \mu  \right)}^{-p'/p}}dyd\tau $$
\end{multline*}
are of the same order in $R$. Herewith we get that $\theta=\alpha$.

Then we have estimate
\begin{equation}\label{2.31}\int\limits_{{\Omega}_{T{{R}^{2/{\alpha}}}}}h(x,t)u^{p}(x,t)\varphi(x,t) dxdt\leq C R^\lambda,\end{equation}
where
$$\lambda={\frac{2}{\alpha}(\alpha-1) p'-2p'-\Biggl(\frac{2}{\alpha}\sigma+\rho\Biggr)\frac{p'}{p}+\frac{2}{\alpha}+N}$$
and
$$C=C(\varepsilon){\int\limits_{\Omega }\Biggl({|\mathcal{D}_{T{R}^{2/\alpha}-,\tau}^{1-\alpha }}({\Delta_y{\Phi }_{\tau}}\circ\mu)|^{p'}}+{{{| ({{\Phi }_{\tau}}\circ \mu )|}^{p'}}}\Biggr)h^{-p'/p}(y,\tau){{\left( \Phi \circ \mu  \right)}^{-p'/p}}dyd\tau.$$

In case $\lambda<0$ (i.e. is $p<p_{c}$) and $R\rightarrow\infty$ in \eqref{2.31}, we have
\begin{equation}\label{2.32}\int\limits_{\Omega }h(x,t)u^{p}(x,t)dxdt\leq0.\end{equation}

Then it follows that $u=0$, which is a contradiction.

If $\lambda=0$, (i.e. is $p=p_{c}$) note that the convergence of the integral in \eqref{2.31} if
$$\Omega_{R}=\{(x,t)\in \mathbb{R}^N\times (0,T): R^2<x^2+t^{\alpha}\leq 2R^2\},$$
then
\begin{equation}\label{2.33}
\lim_{R\rightarrow\infty }\int\limits_{{\Omega}_{R}}h(x,t)u^{p}(x,t)\varphi(x,t)dxdt=0.
\end{equation}

Using the H\"{o}lder inequality we have
\begin{equation}\label{2.34}
\int\limits_{{\Omega}_{T{{R}^{2/\alpha}}}}h(x,t)u^{p}(x,t)\varphi(x,t) dxdt\leq L\biggl(\int\limits_{{\Omega}_{R}}{h(x,t)\arrowvert u(x,t)\mid}^{p}\varphi(x,t)dxdt\biggr)^{1/p},\end{equation}
where
\begin{multline*}L:=\biggl(\int\limits_{{\Omega}_{1}}|\mathcal{D}_{T-,\tau}^{1-\alpha }(\Delta_y{{\Phi }_{\tau }}\circ\mu)|^{p'}h^{-p'/p}(y,\tau)(\Phi \circ\mu)^{-p'/p}dyd\tau\biggr)^{p'}\\+\biggl(\int\limits_{{\Omega}_{1}}{{{\left| ({{\Phi }_{\tau}}\circ \mu )\right|}^{p'}}}h^{-p'/p}(y,\tau){{\left( \Phi \circ \mu  \right)}^{-p'/p}} dyd\tau\biggr)^{1/p'}$$
\end{multline*}
and
$$\Omega_{1}=\{(y,\tau)\in \mathbb{R^N}\times (0,T/R^{2/\alpha}): 1<y^2+\tau^{\alpha}\leq 2\}.$$
Using \eqref{2.34}, we obtain via \eqref{2.33}, after passing to the limit as $R\rightarrow\infty$,
$$\int\limits_{\Omega}h(x,t)u^{p}(x,t) dxdt=0.$$
Then it follows that $u=0$ and completes the proof.

\section{Nonlinear integro-differential diffusion system}
In this section, we show how the method of proof can be used to the integro-differential diffusion system in ${\Omega}_{T}$$({\Omega}_{T}= \mathbb{R^{N}}\times(0,T))$ :
\begin{equation}\label{3.1}\left\{\begin{array}{l}
{u}_{t}( x,t )-{\Delta_x{D_{0+,t}^{1-\alpha }{u}( x,t )}}=h_1(x,t)v^{p}( x,t ) \,\,\,\text{in}\,\,\ {\Omega}_{T},\\{}\\
{v}_{t}( x,t )-{\Delta_x{D_{0+,t}^{1-\beta }{v}( x,t )}}=h_2(x,t)u^{q}( x,t )  \,\,\,\text{in}\,\,\ {\Omega}_{T},\end{array}\right.
\end{equation}
with Cauchy data
\begin{equation}\label{3.2}u\left( x,0 \right)={{u}_{0}}\left( x \right)\ge ~0,\text{   }v\left( x,0 \right)={{v}_{0}}\left( x \right)~\ge 0,\text{ for  }x\in \mathbb{R^{N}},\end{equation}
where $0<\alpha ,\beta <1, h_1(x,t)=|x|^{\rho_1} t^{\sigma_1}, h_2(x,t)=|x|^{\rho_2} t^{\sigma_2}$ and $\sigma_1,\sigma_2>-1$, $\rho_1,\rho_2\geq0.$

\subsection{Local mild solution of integro-differential diffusion system}
In this section we  study the existence of local solutions of \eqref{3.1}-\eqref{3.2}.
\begin{definition} Let $u_0(x), v_0(x)\in C_0(\mathbb{R}^N)$ and $T >0$.

We call $(u,v)\in C_0(\mathbb{R}^N,C[0, T_{\max})\times C_0(\mathbb{R}^N,C[0, T_{\max})$ a mild solution of the problem \eqref{3.1}-\eqref{3.2},  if $(u, v)$ satisfies the following integral equations

\begin{equation}\label{3.3}\left\{\begin{array}{l}u(x,t)=\int_{\mathbb{R}^N}{G\left( x-y,t \right)}{u_0} \left( y \right)dy+\int_{0}^{T}{\int_{\mathbb{R}^N}{G\left( x-y,t-\tau  \right)}f(y,\tau,v)dyd\tau }.
\\{}\\v(x,t)=\int_{\mathbb{R}^N}{G\left( x-y,t \right)}{v_0} \left( y \right)dy+\int_{0}^{T}{\int_{\mathbb{R}^N}{G\left( x-y,t-\tau  \right)}g(y,\tau,u)dyd\tau}.\end{array}\right. \end{equation}
\end{definition}
where $f(y,\tau,v)=|x|^{\rho_1} t^{\sigma_1}v^{p}( x,t )$ and $g(y,\tau,u)= |x|^{\rho_2} t^{\sigma_2}u^{q}( x,t ).$

\begin{theorem} (Local existence).
Let $u_0(x), v_0(x)\in C_0(\mathbb{R}^N)$.

Then, there exists a maximal time $T_{\max} > 0$ such that the problem
\eqref{3.1}-\eqref{3.2} has a unique mild solution $(u, v)\in C_0(\mathbb{R}^N,C[0, T_{\max}))\times C_0(\mathbb{R}^N,C[0, T_{\max}))$. \end{theorem}

Proof.
For arbitrary $T >0$, we define the Banach space
\begin{multline}E_T=\{(u,v)\in C_0(\mathbb{R}^N,C[0, T_{\max}))\times C_0(\mathbb{R}^N,C[0, T_{\max})); \\|||(u,v)|||\leq2(\parallel u_0(x)\parallel_{L^{\infty}(\mathbb{R}^N)}+\parallel v_0(x)\parallel_{L^{\infty}(\mathbb{R}^N)})\}.\end{multline}
where $|||\cdot|||$ is the norm of $E_T$, which represented by
$$|||(u, v)||| = \parallel u\parallel_1 + \parallel v\parallel_1 = \parallel u\parallel_{L^{\infty}(\mathbb{R}^N,L^{\infty}[0,T])} + \parallel v\parallel_{L^{\infty}(\mathbb{R}^N,L^{\infty}[0,T])}.$$

Next, for every $(u, v) \in E_T$, we introduce the map $\Psi$ defined on $E_T$ by
$$\Psi(u, v) = (\Psi_1 (u, v) ,\Psi_2(u, v)),$$
where
$$\Psi_1 (u, v)=\int_{\mathbb{R}^N}{G\left( x-y,t \right)}{u_0} \left( y \right)dy+\int_{0}^{T}{\int_{\mathbb{R}^N}{G\left( x-y,t-\tau  \right)}f(y,\tau,v)dyd\tau }$$
and
$$\Psi_2(u, v)=\int_{\mathbb{R}^N}{G\left( x-y,t \right)}{v_0} \left( y \right)dy+\int_{0}^{T}{\int_{\mathbb{R}^N}{G\left( x-y,t-\tau  \right)}g(y,\tau,u)dyd\tau}.$$

We are going to prove the existence of a unique local solution as a fixed point of $\Psi$ via the Banach fixed point theorem.

\begin{itemize}
  \item $\Psi:E_T\rightarrow E_T.$
\end{itemize}
Let $(u,v)\in E_T;$ using Lemma \eqref{L2.2}, we obtain
\begin{align*}\Psi(u, v)&=\parallel \Psi_1 (u, v)\parallel_{L^{\infty}(\mathbb{R}^N,L^{\infty}[0,T])}
 + \parallel \Psi_2(u, v)\parallel_{L^{\infty}(\mathbb{R}^N,L^{\infty}[0,T])}\\&\leq \parallel u_0(x) \parallel_{L^{\infty}(\mathbb{R}^N)}+\left\|\int_{0}^{T}{\int_{\mathbb{R}^N}{G\left( x-y,t-\tau  \right)}f(y,\tau,v)dyd\tau }\right\|_{L^{\infty}[0,T]}\\&+\parallel v_0(x) \parallel_{L^{\infty}(\mathbb{R}^N)}+\left\|\int_{0}^{T}{\int_{\mathbb{R}^N}{G\left( x-y,t-\tau  \right)}g(y,\tau,u)dyd\tau }\right\|_{L^{\infty}[0,T]}\\&=\parallel u_0(x) \parallel_{L^{\infty}(\mathbb{R}^N)}+\left\|\int_{0}^{T}{\int_{\mathbb{R}^N}{G\left( x-y,t-\tau  \right)}\left\|f(y,\tau,v)\right\|_{L^{\infty}(\mathbb{R}^N)} dyd\tau }\right\|_{L^{\infty}[0,T]}\\&+\parallel v_0(x) \parallel_{L^{\infty}(\mathbb{R}^N)}+\left\|\int_{0}^{T}{\int_{\mathbb{R}^N}{G\left( x-y,t-\tau  \right)}\left\|g(y,\tau,u)\right\|_{L^{\infty}(\mathbb{R}^N)} dyd\tau }\right\|_{L^{\infty}[0,T]}\\&\leq \parallel u_0(x) \parallel_{L^{\infty}(\mathbb{R}^N)}+\parallel v_0(x) \parallel_{L^{\infty}(\mathbb{R}^N)}+T\left\|v\right\|_1+T\left\|u\right\|_1\\&\leq \parallel u_0(x) \parallel_{L^{\infty}(\mathbb{R}^N)}+\parallel v_0(x) \parallel_{L^{\infty}(\mathbb{R}^N)}+2T(\left\|v_0\right\|_{L^{\infty}(\mathbb{R}^N)}+\left\|u_0\right\|_{L^{\infty}(\mathbb{R}^N)}). \end{align*}
 Now, for $T$ such that \,
 $$2T(\left\|v_0(x)\right\|_{L^{\infty}(\mathbb{R}^N)}+\left\|u_0(x)\right\|_{L^{\infty}(\mathbb{R}^N)})\leq \parallel u_0(x) \parallel_{L^{\infty}(\mathbb{R}^N)}+\parallel v_0(x) \parallel_{L^{\infty}(\mathbb{R}^N)},\,\, \Psi(u, v)\in E_T.$$
 \begin{itemize}
  \item $\Psi$ is a contraction.
\end{itemize}
Let $(u,v),\,(\tilde{u},\tilde{v}) \in E_T;$ using Lemma \eqref{L2.2}, we have
\begin{align*}|||\Psi(u, v)&-\Psi(\tilde{u},\tilde{v})|||= \parallel\Psi_1(u, v)-\Psi_1(\tilde{u},\tilde{v})\parallel_1+ \parallel \Psi_2(u, v)-\Psi_2(\tilde{u},\tilde{v})\parallel_1\\&\leq\left\|\int_{0}^{T}{\int_{\mathbb{R}^N}{G\left( x-y,t-\tau  \right)}\left\|f(y,\tau,v)-f(y,\tau,\tilde{v})\right\|_\infty dyd\tau }\right\|_{L^{\infty}[0,T]}\\&+\left\|\int_{0}^{T}{\int_{\mathbb{R}^N}{G\left( x-y,t-\tau  \right)}\left\|g(y,\tau,u)-g(y,\tau,\tilde{u})\right\|_\infty dyd\tau }\right\|_{L^{\infty}[0,T]}\\&\leq T(\left\|f(y,\tau,v)-f(y,\tau,\tilde{v})\right\|_1+\left\|g(y,\tau,u)-g(y,\tau,\tilde{u})\right\|_1)
\\&\leq T(\left\|v-\tilde{v}\right\|_1+\left\|u-\tilde{u}\right\|_1)\\&\leq T|||(u, v)-(\tilde{u},\tilde{v})|||
\\&\leq \frac{1}{2}|||(u, v)-(\tilde{u},\tilde{v})|||,
\end{align*}
where $T$ chosen such that $T\leq\frac{1}{2}.$

Then, by the Banach fixed point theorem, there exists a mild solution $(u, v)\in E_T$ of problem \eqref{3.1}-\eqref{3.2}

\subsection{Critical exponents of Fujita type for the integro-differential diffusion system}
\begin{definition}\label{4.3} A couple of functions $(u,v)\in L^q_{loc}(\mathbb{R^N}\times(0,T))\times L^p_{loc}(\mathbb{R^N}\times(0,T))$
is a weak solution of problem \eqref{3.1}-\eqref{3.2}, if
\begin{multline}\label{4.1}-\int\limits_{\mathbb{R}^N}{u_0( x)\xi (x,0)dx}=\int\limits_{\Omega_T}{u\left( x,t \right){\mathcal{D}}_{T-,t}^{1-\alpha }{\Delta_x{\xi}(x,t)dxdt}}\\+ \int\limits_{\Omega_T}{u\left( x,t \right){{\xi }_{t}}\left( x,t \right)dx dt}+\int\limits_{\Omega_T}{|x|^{\rho_1} t^{\sigma_1}{v^{p}( x,t )\xi ( x,t)dxdt}}\end{multline}
and
\begin{multline}\label{4.2}-\int\limits_{\mathbb{R}^N}{v_0(x)\psi (x,0)dx}=\int\limits_{\Omega_T}{v\left( x,t \right){\mathcal{D}}_{T-,t}^{1-\beta }{\Delta_x{\psi }}(x,t)dxdt}\\ +\int\limits_{\Omega_T}{v\left( x,t \right){{\psi }_{t}}\left( x,t \right)dx dt+}\int\limits_{\Omega_T}{|x|^{\rho_2} t^{\sigma_2}{u^{q}( x,t )\psi ( x,t)dxdt}},\end{multline}
for any test functions  $\xi(x,t),\psi(x,t)\in C_{x,t}^{2,1}( \Omega_{T})$ with $\xi( x,T )=\psi( x,T )=0.$
\end{definition}

\begin{theorem} Let $p>1,\text{ }q>1$. Suppose that
$$1\leq N\leq\max \biggl\{\frac{{l}_{1}+{q{l}_{2}}}{q};{\frac{{p{l}_{1}}+{l}_{2}}{p}\biggr\}} $$
where
$${{l}_{1}}={\frac{2}{\beta}+\frac{1}{p}\left(\frac{2}{\beta}\sigma_1+\rho_1\right)-\frac{1}{p'}\left(\frac{2}{\beta}+N\right)},$$
$${{l}_{2}}={\frac{2}{\alpha}+\frac{1}{q}\left(\frac{2}{\alpha}\sigma_2+\rho_2\right)-\frac{1}{q'}\left({\frac{2}{\alpha}+N}\right)}.$$

Then, the system (with the initial data) does not admit nontrivial global weak nonnegative solutions.
\end{theorem}
Proof. The proof proceeds by contradiction. Therefore, let
$$\xi(x,t)=\Phi\left(\frac{t^{2\theta_1}+|x|^2}{R^2}\right),\,\,\,\,\psi(x,t)=\Phi\left(\frac{t^{2\theta_2}+|x|^2}{R^2}\right)$$
and chosen so that
\begin{multline}\label{4.8}\int\limits_{\Omega_{T{R}^{2/{\theta}_{1} }}}|\mathcal{D}_{\small T{{R}^{2/{\theta}_{1} }}-,t}^{1-\alpha }{\Delta_x{\xi }}(x,t)|^{q'}({h_2\psi })^{-q'/q}(x,t)dxdt<\infty,\\ \int\limits_{\Omega_{T{R}^{2/{\theta}_{1} }}}{{{\left| {{\xi }_{t}}(x,t) \right|}^{q'}}}{({h_2\psi })^{-q'/q}}(x,t)dxdt<\infty,\end{multline}
and
\begin{multline}\label{4.9}\int\limits_{\Omega_{T{R}^{2/{\theta}_{2}}}}|\mathcal{D}_{\small T{{R}^{2/{\theta}_{2} }}-,t}^{1-\beta }{\Delta_x{\psi }(x,t)}|^{p'}(h_1{\xi })^{-p'/p}(x,t)dxdt<\infty,\\ \int\limits_{\Omega_{T{R}^{2/{\theta}_{2}}}}{{{\left| {{\psi }_{t}}(x,t) \right|}^{p'}}}{(h_1{\xi })^{-p'/p}}(x,t)dxdt<\infty.\end{multline}

To estimate the right-hand side of \eqref{4.1} on $\Omega_{T{R}^{2/{\theta}_{1}}}$
\begin{multline*}\int\limits_{{\Omega}_{\small T{{R}^{2/{\theta}_{1} }}}}u(x,t)\mathcal{D}_{\small T{{R}^{2/{\theta}_{1} }}-,t}^{1-\alpha }{\Delta_x{\xi }(x,t)}dxdt\\=\int\limits_{{\Omega}_{T{{R}^{2/{\theta}_{1} }}}}u(x,t)
(h_2\psi)^{1/q}(x,t)[\mathcal{D}_{\small T{{R}^{2/{\theta}_{1} }}-,t}^{1-\alpha }{\Delta_x{\xi }(x,t)}](h_2{\psi})^{-1/q}(x,t)dxdt,\end{multline*}
\begin{equation*}\int\limits_{{\Omega}_{T{{R}^{2/{\theta}_{1}}}}}{u\left( x,t \right){{\xi }_{t}}\left( x,t \right)dx dt}=\int\limits_{{\Omega}_{T{{R}^{2/{\theta}_{1} }}}}u(x,t)
(h_2\psi)^{1/q}(x,t){{\xi }_{t}}(x,t)(h_2{\psi})^{-1/q}(x,t)dxdt.\end{equation*}

Take into consideration the $\varepsilon$-Young equality
$$XY\leq \varepsilon {X}^{q}+C(\varepsilon){Y}^{q'}, \,\,\,\,q+q'=qq',\,\,\,\, X,Y\geq0,$$
we have the estimates

\begin{multline*}\int\limits_{{\Omega}_{T{{R}^{2/{\theta}_{1}}}}}u(x,t)\mathcal{D}_{\small T{{R}^{2/{\theta}_{1}}}-,t}^{1-\alpha }{\Delta_x{\xi }(x,t)}dxdt\leq\\ \varepsilon\int\limits_{{\Omega}_{T{{R}^{2/{\theta}_{1}}}}} |u(x,t)|^{q}(h_2\psi)(x,t)dxdt+C(\varepsilon) \int\limits_{{\Omega}_{T{{R}^{2/{\theta}_{1}}}}}{|\mathcal{D}_{T{{R}^{2/{\theta}_{1}}}-,t}^{1-\alpha }{\Delta_x{\xi }(x,t)}|}^{q'}(h_2\psi)^{-q'/q}(x,t)dxdt\end{multline*}
and
\begin{multline*}\int\limits_{{{\Omega}_{T{{R}^{2/{\theta}_{1}}}}}}u(x,t){\xi}_{t}(x,t)dxdt\\ \le\varepsilon\int\limits_{{\Omega}_{T{{R}^{2/{\theta}_{1}}}}}|u(x,t)|^{q}(h_2\psi)(x,t)dxdt+C(\varepsilon) \int\limits_{{\Omega}_{T{{R}^{2/{\theta}_{1}}}}}{|\xi_{t}(x,t)|}^{q'}(h_2\psi)^{-q'/q}(x,t)dxdt\end{multline*}
respectively.

Similarly,
\begin{multline*}\int\limits_{{\Omega}_{T{{R}^{2/{\theta}_{2}}}}}v(x,t)\mathcal{D}_{\small T{{R}^{2/{\theta}_{2}}}-,t}^{1-\beta }{\Delta_x{\psi }(x,t)}dxdt\leq\\ \varepsilon\int\limits_{{\Omega}_{T{{R}^{2/{\theta}_{2}}}}} |v(x,t)|^{p}(h_1\xi)(x,t)dxdt+C(\varepsilon) \int\limits_{{\Omega}_{T{{R}^{2/{\theta}_{2}}}}}{|\mathcal{D}_{T{{R}^{2/{\theta}_{2}}}-,t}^{1-\beta }{\Delta_x{\psi }(x,t)}|}^{p'}(h_1\xi)^{-p'/p}(x,t)dxdt\end{multline*}
and
\begin{multline*}\int\limits_{{{\Omega}_{T{{R}^{2/{\theta}_{2}}}}}}v(x,t){\psi}_{t}(x,t)dxdt\\ \le\varepsilon\int\limits_{{\Omega}_{T{{R}^{2/{\theta}_{2}}}}}|v(x,t)|^{p}(h_1\xi)(x,t)dxdt+C(\varepsilon) \int\limits_{{\Omega}_{T{{R}^{2/{\theta}_{2}}}}}{|\psi_{t}(x,t)|}^{p'}(h_1\xi)^{-p'/p}(x,t)dxdt.\end{multline*}

Now, taking $\varepsilon$ small enough, we obtain the estimates
\begin{multline}\label{4.10}\int\limits_{{\Omega}_{T{{R}^{2/{\theta}_{1}}}}}{h_2(x,t)\arrowvert u(x,t)\mid}^{q}\psi(x,t)dxdt\\\leq C(\varepsilon) \int\limits_{{\Omega}_{T{{R}^{2/{\theta}_{1} }}}}{\left\{|\mathcal{D}_{T{{R}^{2/{\theta}_{1}}}-,t}^{1-\alpha }{\Delta_x{\xi }(x,t)}|^{q'}+{|\xi_{t}(x,t)\arrowvert}^{q'}\right\}(h_2{\psi })^{-q'/q}(x,t)}dxdt\end{multline}
and
\begin{multline*}\label{4.10}\int\limits_{{\Omega}_{T{{R}^{2/{\theta}_{2}}}}}{h_1(x,t)\arrowvert v(x,t)\mid}^{p}\xi(x,t)dxdt\\ \leq C(\varepsilon) \int\limits_{{\Omega}_{T{{R}^{2/{\theta}_{2} }}}}{\left\{|\mathcal{D}_{T{{R}^{2/{\theta}_{2}}}-,t}^{1-\beta }{\Delta_x{\psi }(x,t)}|^{p'}+{|\psi_{t}(x,t)\arrowvert}^{p'}\right\}(h_1{\xi })^{-p'/p}(x,t)}dxdt.\end{multline*}

Using the H\"{o}lder inequality, we may write
\begin{align*}&{\int\limits_{\Omega_{T{{R}^{2/{\theta}_{1} }}}}{u\left( x,t \right)\left|{\mathcal{D}}_{{T{{R}^{2/{\theta}_{1} }}}-,t}^{1-\alpha }{\Delta_x{\xi }}(x,t)\right| dxdt}} \\& \leq \left({\int\limits_{\Omega_{T{{R}^{2/{\theta}_{1} }}}}{{{\left| u(x,t) \right|^{q}}}(h_2\psi})(x,t)dxdt}\right)^{1/q}\\&\left({\int\limits_{\Omega_{T{{R}^{2/{\theta}_{1} }}}}{{{\left|{\mathcal{D}}_{{T{{R}^{2/{\theta}_{1} }}}-,t}^{1-\alpha }{\Delta_x{\xi }}(x,t)\right|}^{q'}}}(h_2{\psi})^{-q'/q}(x,t)dxdt}\right)^{1/q'}
\end{align*}
and
\begin{multline*}{\int\limits_{\Omega_{T{{R}^{2/{\theta}_{1} }}}}{u\left( x,t \right){{\xi }_{t}}\left( x,t \right)dx dt}} \\ \leq\left({\int\limits_{\Omega_{T{{R}^{2/{\theta}_{1} }}}}{{{\left| u(x,t) \right|^{q}}}(h_2\psi})(x,t)dxdt}\right)^{1/q}\left({\int\limits_{\Omega_{T{{R}^{2/{\theta}_{1}}}}}{{{\left|{{\xi }}_{t}(x,t)\right|}^{q'}}}(h_2{\psi})^{-q'/q}(x,t)dxdt}\right)^{1/q'}.
\end{multline*}
Consequently,
\begin{equation}\label{4.12}\int\limits_{{{\Omega}_{T{{R}^{2/{\theta}_{1}}}}}\text{ }}{{h_1(x,t){\left| v(x,t) \right|}^{p}}}{{\xi }(x,t)dxdt}\leq {{\left( \int\limits_{{{\Omega}_{T{{R}^{2/{\theta}_{1}}}}}}{{{\left| u(x,t) \right|}^{q}}(h_2\psi)(x,t)}dxdt \right)}^{1/q}}\cdot \mathcal{A},\end{equation}
with
\begin{multline*}\mathcal{A}=\left({\int\limits_{\Omega_{T{{R}^{2/{\theta}_{1}}}}}{{{\left|{\mathcal{D}}_{T{{R}^{2/{\theta}_{1}}}-,t}^{1-\alpha }{\Delta_x{\xi }}(x,t)\right|}^{q'}}}(h_2{\psi})^{-q'/q}(x,t)dxdt}\right)^{1/q'}\\+\left({\int\limits_{\Omega_{T{{R}^{2/{\theta}_{1} }}}}{{{\left|{{\xi }}_{t}(x,t)\right|}^{q'}}}(h_2{\psi})^{-q'/q}(x,t)dxdt}\right)^{1/q'} .\end{multline*}
The same way, we obtain the estimate
\begin{equation}\label{4.13}\int\limits_{{{\Omega}_{T{{R}^{2/{\theta}_{2}}}}}\text{ }}{{{h_2(x,t)\left| u(x,t) \right|}^{q}}}{{\psi }(x,t)dxdt}\le {{\left( \int\limits_{{{\Omega}_{T{{R}^{2/{\theta}_{2}}}}}\text{ }}{{{\left| v(x,t) \right|}^{p}}(h_1\xi)(x,t)}dxdt \right)}^{1/p}}\cdot \mathcal{B},\end{equation}
with
\begin{multline*}\mathcal{B}=\left({\int\limits_{\Omega_{T{{R}^{2/{\theta}_{2}}}}}{{{\left|{\mathcal{D}}_{T{{R}^{2/{\theta}_{2}}}-,t}^{1-\beta }{\Delta_x{\psi }}(x,t)\right|}^{p'}}}(h_1{\xi})^{-p'/p}(x,t)dxdt}\right)^{1/p'}\\+\left({\int\limits_{\Omega_{T{{R}^{2/{\theta}_{2} }}}}{{{\left|{{\psi }}_{t}(x,t)\right|}^{p'}}}(h_1{\xi})^{-p'/p}(x,t)dxdt}\right)^{1/p'} .\end{multline*}

In virtue of the inequalities \eqref{4.12} and \eqref{4.13}, we may write
\begin{equation}\label{4.14}{{\left( \int\limits_{{{\Omega}_{T{{R}^{2/{\theta}_{1} }}}}\text{ }}{{h_1(x,t){\left| v(x,t) \right|}^{p}}}{{\xi }(x,t)}dxdt \right)}^{1-\frac{1}{pq}}}\le {{\mathcal{B}}^{1/q}}\cdot \mathcal{A}\end{equation}
and
\begin{equation}\label{4.15}{{\left( \int\limits_{{{\Omega}_{T{{R}^{2/{\theta}_{2}}}}}\text{ }}{{h_2(x,t){\left| u(x,t) \right|}^{q}}}\psi(x,t) dxdt\right)}^{1-\frac{1}{pq}}}\le \mathcal{B}\cdot {{\mathcal{A}}^{1/p}}.\end{equation}

After that, let us make the change of variables $t={R}^{2/{\theta}_{1} }\tau,  x=Ry$  in $\mathcal{A}$.

While  $t={R}^{2/{\theta}_{2} }\tau , x=Ry$ in $\mathcal{B}$, respectively.

And set
$$\Omega_i:=\{(y,\tau)\in {\mathbb{R}^N}\times(0,T/R^{\theta_i}), |y|^2+\tau^{{\theta}_{i}}<2\}, \mu:=\tau^{{\theta}_{i}}+|y|^2, i=1,2.$$

After that, we choose $\theta_1$ such that the right-hand side of $\mathcal{A}$
\begin{align*}&\left(\int\limits_{{\Omega}_{T{{R}^{2/{\theta}_{1}}}}}|\mathcal{D}_{T{{R}^{2/{\theta}_{1}}}-,t}^{1-\alpha }{\Delta_x{\xi }(x,t)}|^{q'}(h_2{\psi })^{-q'/q}(x,t)dxdt\right)^{1/q'}\\&=
\left(\int\limits_{{\Omega }_{T{{R}^{2/{\theta}_{1}}}}}{\left[ -\frac{1}{\Gamma \left( \alpha  \right)}\int\limits_{\tau {{R}^{2/{\theta}_{1}}}}^{T{{R}^{2/{\theta}_{1}}}}{{(s-t)^{\alpha -1 }}}{\Delta_x{\xi }_s}\left( x,s \right)ds \right]}^{q'}{(h_2{\psi })^{-q'/q}}dxdt \right)^{1/q'}\\& \leq R^{\frac{2}{{\theta}_{1}}(\alpha-1) -2-\frac{1}{q} \left({\frac{2}{\theta_1}\sigma_2+\rho_2}\right)+\frac{1}{q'}\left(\frac{2}{{\theta}_{1}}+N\right)}\\&\times\left(\int\limits_{\Omega_1 }|\mathcal{D}_{T-,\tau}^{1-\alpha }(\Delta_y{{\Phi }_{\tau }}\circ\mu)|^{q'}{h_2}^{-q'/q}(y,\tau)(\Phi \circ\mu)^{-q'/q}dyd\tau\right)^{1/q'}\end{align*}
and
\begin{multline*}\left(\int\limits_{{{\Omega}_{T{{R}^{2/{{\theta}_{1}}}}}}}{{{\left| {{\xi }_{\tau}} \right|}^{q'}}}{{(h_2\psi })^{-q'/q}}dxdt\right)^{1/q'}\\ \le{{R}^{-\frac{2}{{\theta}_{1}}-\frac{1}{q} \left({\frac{2}{\theta_1}\sigma_2+\rho_2}\right)+\frac{1}{q'}\left(\frac{2}{{\theta}_{1}}+N\right)}}\left(\int\limits_{\Omega_1 }{{{\left| ({{\Phi }_{\tau}}\circ \mu )\right|}^{q'}}}{h_2}^{-q'/q}(y,\tau){{\left( \Phi \circ \mu  \right)}^{-q'/q}}dyd\tau\right)^{1/q'} \end{multline*}
are of the same order in $R$.

 Herewith we can get $\theta_1=\alpha$. Similarly, $\theta_2=\beta$ in $\mathcal{B}$ .

 Then from \eqref{4.14}-\eqref{4.15} we have the estimates
\[{{\left( \int\limits_{{{\Omega}_{T{{R}^{2/\alpha}}}}\text{ }}{h_1(x,t){{\left| v(x,t) \right|}^{p}}}{{\xi }(x,t)} dxdt\right)}^{1-\frac{1}{pq}}}\le C_1 \left({{R}^{-{{l}_{1}}}}\right)^{1/q}{{R}^{-{{l}_{2}}}} \]
and
$${{\left( \int\limits_{{{\Omega}_{T{{R}^{2/\beta}}}}\text{ }}{h_2(x,t){{\left| u(x,t) \right|}^{q}}}\psi(x,t) dxdt\right)}^{1-\frac{1}{pq}}}\le C_2 {{R}^{-{{l}_{1}}}} \left({{R}^{-{{l}_{2}}}}\right)^{1/p} ,$$
where
${{l}_{1}}={\frac{2}{\beta}+\frac{1}{p}\left(\frac{2}{\beta}\sigma_1+\rho_1\right)-\frac{1}{p'}\left(\frac{2}{\beta}+N\right)},$    ${{l}_{2}}={\frac{2}{\alpha}+\frac{1}{q}\left(\frac{2}{\alpha}\sigma_2+\rho_2\right)-\frac{1}{q'}\left({\frac{2}{\alpha}+N}\right)}.$

Accordingly,
\begin{equation}\label{4.16}{{\left( \int\limits_{{{\Omega}_{T{{R}^{2/\alpha}}}}\text{ }}{h_1(x,t){{\left| v(x,t) \right|}^{p}}}{{\xi }(x,t)} dxdt\right)}^{1-\frac{1}{pq}}}\le C_1{{R}^{-\left( {{l}_{1}}/q+{{l}_{2}} \right)}}\end{equation}
and
\begin{equation}\label{4.17}{{\left( \int\limits_{{{\Omega}_{T{{R}^{2/\beta}}}}\text{ }}{h_2(x,t){{\left| u(x,t) \right|}^{q}}}{{\psi }(x,t)} dxdt\right)}^{1-\frac{1}{pq}}}\le C_2{{R}^{-\left( {{l}_{1}}+{{l}_{2}}/p \right)}},\end{equation}
where
\begin{multline*}C_1=C(\varepsilon)\Biggl[\int\limits_{\Omega_1 }\left(|\mathcal{D}_{TR^{2/\beta}-,t}^{1-\beta}({\Delta_y{\Phi }_{\tau }}\circ\mu)|^{p'}+{{{| ({{\Phi }_{\tau}}\circ \mu )|}^{p'}}}\right)h_1^{-p'/p}(y,\tau){{\left( \Phi \circ \mu  \right)}^{-p'/p}}dyd\tau\Biggr]^{1/q}\\ \times  \int\limits_{\Omega_1}\left(|\mathcal{D}_{TR^{2/\alpha}-,t}^{1-\alpha }(\Delta_y{{\Phi }_{\tau}}\circ\mu)|^{q'}+{{{| ({{\Phi }_{\tau}}\circ \mu )|}^{q'}}}\right)h_2^{-q'/q}(y,\tau){{\left( \Phi \circ \mu  \right)}^{-q'/q}}dyd\tau,\end{multline*}
\begin{multline*}C_2=C(\varepsilon)\int\limits_{\Omega_1 }\left(|\mathcal{D}_{TR^{2/\beta}-,t}^{1-\beta}(\Delta_y{{\Phi }_{\tau }}\circ\mu)|^{p'}+{{{| ({{\Phi }_{\tau}}\circ \mu )|}^{p'}}}\right)h_1^{-p'/p}(y,\tau){{\left( \Phi \circ \mu  \right)}^{-p'/p}}dyd\tau\\ \times\Biggl[\int\limits_{\Omega_1 }\left(|\mathcal{D}_{TR^{2/\alpha}-,t}^{1-\alpha }(\Delta_y{{\Phi }_{\tau }}\circ\mu)|^{q'}+{{{| ({{\Phi }_{\tau}}\circ \mu )|}^{q'}}}\right)h_2^{-q'/q}(y,\tau){{\left( \Phi \circ \mu  \right)}^{-q'/q}}dyd\tau\Biggr]^{1/p}.\end{multline*}

If we choose ${-\left( {{l}_{1}}/q+{{l}_{2}} \right)}<0$ in \eqref{4.16} and let $R\rightarrow\infty$, we obtain
\begin{equation}\label{4.18}{\left( \int\limits_{\Omega_{TR^{2/\alpha}}}h_1(x,t){\left|v(x,t)\right|}^{p}{{\xi }(x,t)} dxdt\right)}^{1-\frac{1}{pq}}\leq0.\end{equation}

This implies that $v=0$ a.e., which is a contradiction.

In case ${-\left( {{l}_{1}}/q+{{l}_{2}} \right)}=0$, observe that the convergence of the integral in \eqref{4.16} if
$$\Omega_{R}=\{(x,t)\in {\mathbb{R}^N}\times (0,T): R^2<x^2+t^{\alpha}\leq 2R^2\}$$
then
\begin{equation}\label{4.19}
\lim_{R\rightarrow\infty} \int\limits_{{\Omega}_{R}}h_1(x,t){\arrowvert v(x,t)\mid}^{p}{{\xi }(x,t)} dxdt=0.
\end{equation}

Using the H\"{o}lder inequality in place of estimate \eqref{4.10}, we have
\begin{equation}\label{4.20}
\int\limits_{{\Omega}_{T{{R}^{2/\alpha }}}}h_1(x,t){\arrowvert v(x,t)\mid}^{p}\xi(x,t) dxdt\leq L\biggl(\int\limits_{{\Omega}_{R}}h_1(x,t){\arrowvert v(x,t)\mid}^{p}\xi(x,t) dxdt \biggr)^{1/p},\end{equation}
where
\begin{multline*}L:=\biggl(\int\limits_{{\Omega}^{1}}|\mathcal{D}_{T{{R}^{2/\alpha }}-,t}^{1-\alpha }(\Delta_y{{\Phi }_{\tau}}\circ\mu)|^{p'}h_1^{-p'/p}(y,\tau)(\Phi \circ\mu)^{-p'/p}\biggr)^{p'}\\+\biggl(\int\limits_{{\Omega}^{1}}{{{\left| ({{\Phi }_{\tau}}\circ \mu )\right|}^{p'}}}h_1^{-p'/p}(y,\tau){{\left( \Phi \circ \mu  \right)}^{-p'/p}} dyd\tau\biggr)^{1/p'}\end{multline*}
and
$$\Omega^{1}=\{(y,\tau)\in {\mathbb{R}^N}\times (0,T/R^{2/\alpha}): 1<y^2+\tau^{\alpha}\leq 2\}.$$

According to \eqref{4.20}, we obtain via \eqref{4.19}, after passing to the limit as $R\rightarrow\infty$,
$$\int\limits_{\Omega_1}h_1(x,t){\arrowvert v(x,t)\mid}^{p} dxdt=0.$$

This leads to $v(x,t)=0$.

In the case ${{l}_{1}}/q+{{l}_{2}}\ge 0$ is
\begin{equation}\label{4.21}N\le \frac{2(\alpha(1+\sigma_1)+p\beta(1+\sigma_2))+\alpha\beta(\rho_1+p\rho_2)}{\alpha\beta(pq-1)}\end{equation}
and according to \eqref{4.15}, we get
\begin{equation}\label{4.22}N\le \frac{2(\beta(1+\sigma_2)+q\alpha(1+\sigma_1))+\alpha\beta(p\rho_1+\rho_2)}{\alpha\beta(pq-1)}.\end{equation}

Consequently, from \eqref{4.21} and \eqref{4.22} we obtain

\begin{multline*}1\leq N\leq\max \biggl\{\frac{2(\alpha(1+\sigma_1)+p\beta(1+\sigma_2))+\alpha\beta(\rho_1+p\rho_2)}{\alpha\beta(pq-1)},
\\ \frac{2(\beta(1+\sigma_2)+q\alpha(1+\sigma_1))+\alpha\beta(p\rho_1+\rho_2)}{\alpha\beta(pq-1)}\biggr\}.\end{multline*}

\section*{\large Acknowledgments}
This research is financially supported in parts by the FWO Odysseus Project, as well as by the Ministry of Education and Science of the Republic of Kazakhstan Grant No. AP05131756. No new data was collected or generated during the course of research.

\end{document}